\documentclass[11pt]{amsart}
\usepackage{amsfonts}
\usepackage{amssymb}
\usepackage{amsmath}
\numberwithin{equation}{section}
\usepackage{dsfont}
\usepackage{color}
\usepackage{graphicx}
\newtheorem{theorem}{Theorem}[section]
\newtheorem{lemma}[theorem]{Lemma}
\newtheorem{proposition}[theorem]{Proposition}
\newtheorem{corollary}[theorem]{Corollary}

\newtheorem{observation}[theorem]{Observation}
\theoremstyle{definition}
\newtheorem{df}{Definition}
\newtheorem{example}[df]{Example}
\newtheorem{remark}[df]{Remark}

\newcommand{\N}{\mathbb N}

\newcommand{\R}{\mathbb R}
\newcommand{\Z}{\mathbb Z}
\newcommand{\ve}{\varepsilon}

\makeatletter
\@namedef{subjclassname@2020}{%
  \textup{2020} Mathematics Subject Classification}
\makeatother


\begin{document}
\author{Piotr Nowakowski}
\address{Faculty of Mathematics and Computer Science, University of \L \'{o}d\'{z},
Banacha 22, 90-238 \L \'{o}d\'{z}, Poland}
\address{Institute of Mathematics, Czech Academy of Sciences,
\v{Z}itn\'a 25, 115 67 Prague 1, Czech Republic\\
ORCID: 0000-0002-3655-4991}
\email{piotr.nowakowski@wmii.uni.lodz.pl}

\title{Characterization of the algebraic difference of special affine Cantor sets}
\date{}

\begin{abstract}
We investigate some self-similar Cantor sets $C(l,r,p)$, which we call S-Cantor sets, generated by numbers $l,r,p \in \N$, $l+r<p$. We give a full characterization of the set $C(l_1,r_1,p)-C(l_2,r_2,p)$ which can take one of the form: the interval $[-1,1]$, a Cantor set, an L-Cantorval, an R-Cantorval or an M-Cantorval. As corollaries we give examples of Cantor sets and Cantorvals, which can be easily described using some positional numeral systems.
\end{abstract}

\maketitle

\section{Introduction}
We denote by $A\pm B$ the set $\left\{ a\pm
b:a\in A,\,b\in B\right\} $, where $A,B\subset \mathbb{R}$. The set $A- B$ is called the algebraic difference of sets $A$ and $B$. The set $A-A$ is called the difference set of a set $A$. We will also write $a+A$ rather than $\{a\} + A$ for $a \in \R$. 
Let $I \subset \R$ be an interval. We denote by $l(I)$, $r(I)$ the left and the right endpoint of $I$, respectively,.

We say that a set $C\subset \R$ is a Cantor set if it is nonempty, compact, perfect and nowhere dense.

For a set $C \subset \R$, every 
component of 
the set $\R \setminus C$ is called a gap of $C$. A component of $C$ is called proper if it is not a singleton.

Let us recall the definitions of three types of Cantorvals (compare \cite{MO}). A perfect set $E \subset \R$ is called an M-Cantorval if it has infinitely many gaps and both endpoints of any gap are accumulated by gaps and proper components of $E$.
A perfect set $E \subset \R$ is called an L-Cantorval (respectively, an R-Cantorval) if it has infinitely many gaps, the left (right) endpoint of any gap is accumulated by gaps and proper components of $E$, and the right (left) endpoint of any gap is an endpoint of a proper component of $E$.

Many authors were examining algebraic differences or sums of various types of Cantor sets (e.g.  \cite{AC}, \cite{FN}, \cite{HKY}, \cite{MY}, \cite{N}, \cite{P}, \cite{FF}, \cite{S}, \cite{Sol}). One can find its application for example in spectral theory (see \cite{DGS},\cite{T16}) or number theory (see \cite{H}).

The algebraic difference (or sum) of dynamically defined Cantor sets is of special interest. The question about the structure of algebraic difference of such sets appear naturally in the studies of homoclinic bifurcations of dynamical systems (see \cite{PT}). In \cite{MO}, the classes of so-called regular, affine and homogeneous Cantor sets were considered. Every homogeneous Cantor set is affine, and every affine is regular. Affine Cantor sets are also self-similar. There was proved that the algebraic sum of two homogeneous Cantor sets is either a Cantor set, a finite union of closed intervals, an L-Cantorval, an R-Cantorval or an M-Cantorval. In this paper the authors posed three open questions:

1. Does the similar result hold generically for affine Cantor sets?  

2. What can be said about the topological structure of the sum of two
regular Cantor sets?

3. Is it possible to characterize each one of the five possibilities for the algebraic sum of homogeneous Cantor sets?

In \cite{ACI}, the positive answer to the first question was given. As far as we know, there are still no answers to the remaining questions. However, studies on homogeneous and affine Cantor sets are still conducted by many researchers (e.g. \cite{P}, \cite{T19}, \cite{E}).

In our paper we present a class of S-Cantor subsets of $[0,1]$. They are affine (in the sense given in \cite{MO}), but not homogeneous. For such sets we prove not only that the algebraic sum can be either a Cantor set, a finite union of closed intervals, an L-Cantorval, an R-Cantorval or an M-Cantorval, but we also give a full characterization when every class occurs. It answers questions analogous to questions 2. and 3. mentioned above for this particular class of Cantor sets.

\section{The algebraic difference of S-Cantor sets}

In this section we consider another type of Cantor sets which was examined for example in \cite{R}, \cite{Z}, \cite{KRC}.
We call them $p$-Cantor sets. In the class of $p$-Cantor sets we will distinguish the subclass of S-Cantor sets. But first, we need some additional notation. Let $A \subset \Z$ be a finite set, $p \in \N$, $p \geq 2$. Set
$$A_p:=\left\{\sum_{i=1}^{\infty}\frac{x_i}{p^i}\colon x_i\in A\right\}.$$
This notation comes from \cite{R}. For example, we have 
$\Z(p)_p = [0, 1]$, $\{-p+1,-p+2, \dots, p-2, p-1\}_p = [-1,1]$, $\{0\}_p =\{0\}$, $\{p - 1\}_p = \{1\}$, $\{0, 2\}_3 = C$,
where $C$ is the classical Cantor ternary set, and $\Z(p) := \{0, 1, \dots, p - 1\}$.
Note that, if $A\subset\Z(p)$, the set $A_p$ consists of all real numbers in $[0,1]$ having $p$-adic expansions 
with all digits in $A$.

We will now give some easy but useful properties of the sets $A_p$.
\begin{proposition} \label{dod}
Let $A$ and  $B$ be finite subsets of $\Z$, $p \in \N,$ $p \geq 2,$ $k \in \Z.$ Then:
\begin{itemize}
\item[\mbox{(i)}]$ A_p + B_p = (A + B)_p;$
\item[\mbox{(ii)}]$A_p - B_p = (A - B)_p;$
\item[\mbox{(iii)}]$kA_p =(kA)_p.$
\end{itemize}
\end{proposition}
\begin{proof}
(i) We have
$$A_p + B_p =\left\{\sum_{i=1}^{\infty}\frac{x_i}{p^i}\colon x_i\in A\right\} + \left\{\sum_{i=1}^{\infty}\frac{y_i}{p^i}\colon y_i\in B\right\} =\left\{\sum_{i=1}^{\infty}\frac{x_i}{p^i}+\sum_{i=1}^{\infty}\frac{y_i}{p^i}\colon x_i\in A,y_i \in B\right\}$$
$$= \left\{\sum_{i=1}^{\infty}\frac{x_i+y_i}{p^i}\colon x_i\in A,y_i \in B\right\}=\left\{\sum_{i=1}^{\infty}\frac{z_i}{p^i}\colon z_i\in A+B\right\}=(A + B)_p.$$
The proof of (ii) and (iii) is analogous. 
\end{proof}

The next proposition is a mathematical folklore (see, e.g. \cite{R}), so we omit the proof, which is quite technical.
\begin{proposition} \label{P1}
Let $A$ be a finite subset of $\Z$ and $p \in \N, p \geq 2$. Then:
\begin{itemize}
\item[\mbox{(i)}]  $A_p$ is a closed set.
\item[\mbox{(ii)}] If $A$ has more than $1$ element, then $A_p$ is a perfect set.
\item[\mbox{(iii)}] If $A \subsetneq \Z(p)$, then $A_p$ is nowhere dense.
\item[\mbox{(iv)}] If $A \subsetneq \Z(p)$ and $A$ has more than $1$ element, then $A_p$ is a Cantor set.
\end{itemize}
\end{proposition}

For $i, j \in\Z,$ we set 
$$\langle i, j\rangle := \left\{ \begin{array}{ccc}
[i, j]\cap \Z \;\text{ if }\; i < j \\ 
\{i\} \;\text{ if }\; i=j\\
\emptyset \;\text{ if }\; i>j.%
\end{array}%
\right. $$

If $(x_i) \in \langle -p, p \rangle^{n}$, then we will write $\bar{x}_n = \sum_{i=1}^n \frac{x_i}{p^i},$ for $n \in \N$. 

In the sequel we will consider sets $A_p$, where $p \in \N$, $p>2$, $A \subset \langle -p+1,p-1 \rangle$ and $-p+1,0,p-1 \in A$. Note that $A_p \subset [-1,1]$. 
Let $n \in \N, x \in [-1,1].$ We say that $x$ is $(n)$-bi-obtainable if for some $k \in \langle -p^n, p^n-1 \rangle$ we have $x \in [\frac{k}{p^n},\frac{k+1}{p^n}]$ and there exist sequences $(y_i),(z_i) \in A^{n}$ such that $\overline{y}_n = \frac{k}{p^n}$ and $\overline{z}_n = \frac{k}{p^n} + \frac{1}{p^n}$. 

\begin{remark} \label{ibo}
Observe that the condition stating that there exist sequences $(y_i),(z_i) \in A^{n}$ such that $\overline{y}_n = \frac{k}{p^n}$ and $\overline{z}_n = \frac{k}{p^n} + \frac{1}{p^n}$ implies that $\frac{k}{p^n}, \frac{k}{p^n} + \frac{1}{p^n} \in A_p$. Indeed, 
$$\frac{k}{p^n} = \sum_{i=1}^n\frac{y_i}{p^i}+\sum_{i=n+1}^{\infty} \frac{0}{p^i} \in A_p.$$ 
Similarly, 
$$\frac{k}{p^n}+ \frac{1}{p^n} = \sum_{i=1}^n\frac{z_i}{p^i}+\sum_{i=n+1}^{\infty} \frac{0}{p^i} \in A_p.$$ 
\end{remark}

We will often use the following easy observations.
\begin{observation} \label{wiel}
\begin{itemize}
\item[(i)] For any $ (x_n) \in \langle -p, p \rangle ^{\N}$ and any $n \in \N$ we have 
$$\overline{x}_n = \frac{k}{p^n}\mbox{ for some }k \in \Z.$$ 

\item[(ii)] For every $n \in \N,k \in \langle -p^n , p^n \rangle$ there exists a sequence $(w_i) \in  \langle -p, p \rangle^n$  such that $\overline{w}_n = \frac{k}{p^n}$.

\item[(iii)] For every $n \in \N, k \in \langle -p^n+1 , p^n-1 \rangle$ there exists a sequence $(w_i) \in  \langle -p+1, p-1 \rangle^n$ such that $\overline{w}_n = \frac{k}{p^n}$. 

\item[(iv)] If $x \in [0,\frac{1}{p^n}]$, then there is a sequence $(w_i) \in  \langle 0, p-1 \rangle^{\N}$ such that $\sum_{i=n+1}^{\infty} \frac{w_i}{p^i} = x$. 

\item[(v)] If $x \in [-\frac{1}{p^n},0]$, then there is a sequence $(w_i) \in  \langle -p+1, 0 \rangle^{\N}$ such that $\sum_{i=n+1}^{\infty} \frac{w_i}{p^i} = x$. 
\end{itemize}
\end{observation}

We will also need the following technical lemma.
\begin{lemma} \label{Lpom}
Let $p \in \N$, $p>2$, $A \subset \langle -p+1,p-1 \rangle$ and $-p+1,0,p-1 \in A$.
\begin{itemize}
\item[\mbox{(1)}]  If $x \in A_p,$ $(x_i) \in A^{\N}$ is such that $x=\sum_{i=1}^{\infty}\frac{x_i}{p^i}$ and, for some $n \in \N$ and $(y_i) \in \langle-p, p-1 \rangle^n$, $x \in (\overline{y}_n,\overline{y}_n+ \frac{1}{p^n})$, then $\overline{x}_n = \overline{y}_n$ or $\overline{x}_n = \overline{y}_n + \frac{1}{p^n}$.
Moreover, 

\item[\mbox{(1.1)}] if $\overline{x}_n =\overline{y}_n$ and $y_n > 0$, then $x_n = y_n$ or $x_n = -p+y_n$;

\item[\mbox{(1.2)}] if $\overline{x}_n = \bar{y}_n$ and $y_n < 0$, then $x_n = y_n$ or $x_n = p+y_n$;

\item[\mbox{(1.3)}] if $\overline{x}_n = \overline{y}_n + \frac{1}{p^n}$ and $y_n > 0$, then $x_n = y_n+1$ or $x_n = -p+1+y_n$;

\item[\mbox{(1.4)}] if $\overline{x}_n = \overline{y}_n + \frac{1}{p^n}$ and $y_n < 0$, then $x_n = y_n +1$ or $x_n = p+y_n+1$.

\item[\mbox{(1.5)}] if $\overline{x}_n = \overline{y}_n$ and $y_n = 0$, then $x_n = 0$.

\item[\mbox{(2)}] If $(x,y)$ is a gap in $A_p$, $(x_i), (y_i) \in A^{\N}$ are such that $x=\sum_{i=1}^{\infty}\frac{x_i}{p^i}$ and $y=\sum_{i=1}^{\infty}\frac{y_i}{p^i}$, then there are $n,m \in \N$ such that $x_n+1,y_m-1 \notin A$ and $x_i = p-1, y_j=-p+1$ for $i > n$, $j > m$. 

\item[\mbox{(2')}] If $ w \in [-1,1]$ is such that $w = \overline{v}_k$ for some $k \in \N$ and $(v_i) \in A^{k}$, then $w$ is not an endpoint of a gap in $A_p$. 

\item[\mbox{(3)}] Let $x \in [-1,1]$. If there exists $n \in \N$ such that $x$ is $(i)$-bi-obtainable for $i \geq n$, then $x \in A_p$.

\item[\mbox{(4)}] Let $x \in [-1,1]$. If for any $k \in \langle 0, p-1 \rangle$ we have $k \in A$ or $k-p \in A$ and there is $n \in \N$ such that $x$ is $(n)$-bi-obtainable, then $x$ is $(i)$-bi-obtainable for $i \geq n$.
\end{itemize}
\end{lemma}
\begin{proof}
Ad (1) By the assumption, we have $0 < x - \overline{y}_n < \frac{1}{p^n}.$ Observe that 
$$|x-\overline{x}_n| = \left\vert \sum_{i=n+1}^{\infty} \frac{x_i}{p^i}\right\vert \leq \frac{p-1}{p^{n+1}(1-\frac{1}{p})}= \frac{1}{p^{n}},$$ 
and so $-\frac{1}{p^n} \leq \overline{x}_n - x \leq \frac{1}{p^n}$. Adding the obtained inequalities, we receive $-\frac{1}{p^n} < \overline{x}_n - \overline{y}_n < \frac{2}{p^n}$. By Observation \ref{wiel}, we have $\overline{x}_n - \overline{y}_n = 0$ or $\overline{x}_n-\overline{y}_n= \frac{1}{p^n}$.

Ad (1.1) Since $y_n > 0$, we have $\overline{y}_n > \overline{y}_{n-1}$. Moreover, $$\overline{y}_n + \frac{1}{p^n} = \overline{y}_{n-1} + \frac{y_n+1}{p^n} \leq \overline{y}_{n-1} + \frac{p}{p^{n}} = \overline{y}_{n-1} + \frac{1}{p^{n-1}}.$$ Hence $x \in (\overline{y}_{n-1}, \overline{y}_{n-1} + \frac{1}{p^{n-1}}).$ By (1), $\overline{x}_{n-1} = \overline{y}_{n-1}$ or $\overline{x}_{n-1} = \overline{y}_{n-1} + \frac{1}{p^{n-1}}.$
In the first case, we have $x_n = y_n$, and in the second case, $\frac{x_n}{p^n} = \frac{y_n}{p^n} - \frac{1}{p^{n-1}}$, so $x_n = y_n-p$. 

Ad (1.2) Since $y_n < 0$, we have $\overline{y}_n < \overline{y}_{n-1}$, and so $\overline{y}_n + \frac{1}{p^n} \leq \overline{y}_{n-1}.$ Moreover, $$\overline{y}_n = \overline{y}_{n-1} + \frac{y_n}{p^n} \geq \overline{y}_{n-1} - \frac{p}{p^{n}} = \overline{y}_{n-1} - \frac{1}{p^{n-1}}.$$ Hence $x \in (\overline{y}_{n-1} - \frac{1}{p^{n-1}}, \overline{y}_{n-1}).$ By Observation \ref{wiel}, there is a sequence $(y'_i)\in \langle -p, p-1 \rangle^{n-1}$ such that $\overline{y'}_{n-1}=\overline{y}_{n-1} - \frac{1}{p^{n-1}}$. By (1), $$\overline{x}_{n-1} = \overline{y'}_{n-1} + \frac{1}{p^{n-1}} = \overline{y}_{n-1}$$ or $$\overline{x}_{n-1} = \overline{y'}_{n-1} = \overline{y}_{n-1} - \frac{1}{p^{n-1}}.$$
In the first case, we have $x_n = y_n$, and in the second case, $\frac{x_n}{p^n} = \frac{y_n}{p^n} + \frac{1}{p^{n-1}}$, so $x_n = y_n+p$. 

The proofs of (1.3) and (1.4) are similar.

Ad (1.5) We will show that $\overline{x}_{n-1} = \overline{y}_{n-1}.$ Indeed, if $|\overline{x}_{n-1} - \overline{y}_{n-1}| > 0$, then, by Observation \ref{wiel}, $|\overline{x}_{n-1} - \overline{y}_{n-1}| \geq \frac{1}{p^{n-1}},$ and hence $$\left\vert\overline{x}_{n} -\overline{y}_{n}\right\vert =\left\vert\overline{x}_{n-1} - \overline{y}_{n-1} + \frac{x_n}{p^n} - \frac{y_n}{p^n}\right\vert \geq \left\vert\overline{x}_{n-1} - \overline{y}_{n-1}\right\vert - \left\vert\frac{x_n}{p^n} - \frac{y_n}{p^n}\right\vert \geq \frac{1}{p^{n-1}} - \frac{p-1}{p^n} = \frac{1}{p^n} > 0.$$ So, $\overline{x}_{n-1} = \overline{y}_{n-1},$ and thus $x_n = y_n$.

Ad (2) First, we will show that there is $n \in \N$ such that $x_i = p-1$ for any $i > n$. On the contrary, assume that for any $k \in \N$ there is $j > k$ such that $x_j < p-1$. Let $k\in \N$ be such that $\frac{2}{p^k} < y-x$ and let $j > k$ be such that $x_j < p-1$. Let $z_i = x_i$ for $i \neq j$, and $z_j = p-1$. Since $p-1 \in A$ and $x \in A_p$, we have $z = \sum_{i=1}^{\infty} \frac{z_i}{p^i} \in A_p$. Moreover, 
$$z = x + \frac{p-1-x_j}{p^j} \leq x + \frac{2(p-1)}{p^j} < x + \frac{2}{p^{j-1}} \leq  x + \frac{2}{p^{k}}< y.$$ 
So, $z \in A_p \cap(x,y),$ a contradiction. Of course, $x < 1$, so there is $n \in \N$ such that $x_n < p-1$ and $x_i = p-1$ for $i > n$. 

Now, we will show that $x_n+1 \notin A$.
On the contrary, assume that $x_n+1 \in A$. We have $$x =  \sum_{i=1}^{\infty}\frac{x_i}{p^i} = \overline{x}_{n}+\sum_{i=n+1}^{\infty}\frac{p-1}{p^i}=\overline{x}_{n} + \frac{1}{p^n}.$$ Let $k > n$ be such that $\frac{1}{p^{k-1}} < y-x$. Let $w_i = x_i$ for $i <n$, $w_n = x_{n}+1$, $w_k = p-1$ and $w_i = 0$ for the remaining $i$. Then $w  =\sum_{i=1}^{\infty}\frac{w_i}{p^i} \in A_p$ and $$w = \overline{x}_{n} + \frac{1}{p^n} + \frac{p-1}{p^k} < x + \frac{1}{p^{k-1}} < y,$$ a contradiction. Thus, $x_n+1 \notin A$. The proof for the sequence $(y_i)$ is analogous.

Ad (2') Immediately follows from (2).

Ad (3) Using the definition of $(i)$-bi-obtainability and Remark \ref{ibo}, we can find $k_i \in \langle -p^i, p^i-1 \rangle$ such that $x \in [\frac{k_i}{p^i},\frac{k_i}{p^i}+\frac{1}{p^i}]$ and $\frac{k_i}{p^i} \in A_p$ for any $i \geq n$. Then $|x-\frac{k_i}{p^i}| \leq \frac{1}{p^i}$, so $\lim\limits_{i \to \infty} \frac{k_i}{p^i} = x$. Since $A_p$ is closed, then $x \in A_p$.

Ad (4) It suffices to show that if $x$ is $(n)$-bi-obtainable, then it is $(n+1)$-bi-obtainable. 

First, suppose that $x = \frac{k}{p^n}$ for some $k \in \langle -p^{n}+1, p^{n}-1 \rangle$ (of course, $1$ and $-1$ cannot be $(n)$-bi-obtainable). Then, from the definition of $(n)$-bi-obtainability it follows that $\frac{k}{p^n} \in A_p$ and at least one of $\frac{k}{p^n} + \frac{1}{p^n}$ or $\frac{k}{p^n} - \frac{1}{p^n}$ belongs to $A_p$. Assume that $\frac{k}{p^n} + \frac{1}{p^n} \in A_p$ (the proof, when $\frac{k}{p^n} - \frac{1}{p^n} \in A_p$ is similar). Let $(y_i),(z_i) \in A^{n}$ be such that $\overline{y}_{n} = \frac{k}{p^n}$ and $\overline{z}_{n} = \frac{k}{p^n} + \frac{1}{p^n}$. We have $x \in [\frac{kp}{p^{n+1}},\frac{kp}{p^{n+1}}+\frac{1}{p^{n+1}}]$ and $\frac{kp}{p^{n+1}} = \frac{k}{p^n} \in A_p$.  We need to find a sequence $(w_i) \in \langle -p, p-1 \rangle ^{n+1}$ such that $\overline{w}_{n+1} = \frac{k}{p^{n}}+\frac{1}{p^{n+1}}$. 
Consider the cases.

$1^{\mbox{o}}$  $1 \in A$. Put $w_i = y_i$ for $i \leq n$ and $w_{n+1} = 1$. Then $(w_i) \in A^{n+1}$ and $\overline{w}_{n+1} = \frac{k}{p^{n}}+\frac{1}{p^{n+1}}$.

$2^{\mbox{o}}$  $1 \notin A$. By the assumption, we have $1-p \in A$. Put $w_i = z_i$ for $i \leq n$ and $w_{n+1} = 1 - p $. Then $(w_i) \in A^{n+1}$ and $$\overline{w}_{n+1} = \frac{k}{p^{n}}+\frac{1}{p^{n}} + \frac{1-p}{p^{n+1}} =\frac{k}{p^{n}} + \frac{1}{p^{n+1}}.$$

Therefore, $x$ is $(n+1)$-bi-obtainable.

Now, assume that $x \in (\frac{k}{p^{n}},\frac{k}{p^{n}}+\frac{1}{p^{n}})$ for some $k\in \langle -p^{n}, p^{n}-1 \rangle$. Then there is $k' \in \langle 0, p-1\rangle$ such that $x \in [\frac{kp+k'}{p^{n+1}},\frac{kp+k'}{p^{n+1}} +\frac{1}{p^{n+1}}]$.   
Since $x$ is $(n)$-bi-obtainable, there exist $(y_i),(z_i) \in A^{n}$ such that $\overline{y}_{n} = \frac{k}{p^{n}}$ and $\overline{z}_{n} = \frac{k}{p^{n}} + \frac{1}{p^n}$.
We will find a sequence $(v_i) \in A^{n+1}$ such that $\overline{v}_{n+1} = \frac{kp+k'}{p^{n+1}}$.
Consider the cases.

$1^{\mbox{o}}$  $k' \in A$. Put $v_i = y_i$ for $i \leq n$ and $v_{n+1} = k'$. Then $(v_i) \in A^{n+1}$ and $\overline{v}_{n+1} = \frac{kp+k'}{p^{n+1}}$.

$2^{\mbox{o}}$  $k' \notin A$. Since $k' \geq 0,$ by the assumption, we have $k'-p \in A$. Put $v_i = z_i$ for $i \leq n$ and $v_{n+1} = k' - p $. Then $(v_i) \in A^{n+1}$ and $$\overline{v}_{n+1} = \frac{k}{p^{n}} + \frac{1}{p^n} + \frac{k'-p}{p^{n+1}} = \frac{kp+k'}{p^{n+1}}.$$

Now, we will find a sequence $(u_i) \in A^{n+1}$ such that $\overline{u}_{n+1} = \frac{kp+k'}{p^{n+1}} + \frac{1}{p^{n+1}}$.
Consider the cases.

$1^{\mbox{o}}$  $k' + 1 \in A$. Then put $u_i = y_i$ for $i \leq n$ and $u_{n+1} = k' + 1$. Then $(u_i) \in A^{n+1}$ and $\overline{u}_{n+1} = \frac{kp+k'}{p^{n+1}} + \frac{1}{p^{n+1}}$.

$2^{\mbox{o}}$  $k' + 1 \notin A$. If $k' < p-1$, then $k' +1 \in \langle 1, p-1 \rangle$ and, by the assumption, $k'+1-p \in A$. If $k' = p-1$, then $k'+1-p = 0 \in A$. In both cases put $u_i = z_i$ for $i \leq n$ and $u_{n+1} = k' +1 - p $. Then $(u_i) \in A^{n+1}$ and $$\overline{u}_{n+1} =  \frac{k}{p^{n}}+ \frac{1}{p^n} + \frac{k'-p+1}{p^{n+1}} = \frac{kp+k'}{p^{n+1}}+ \frac{1}{p^{n+1}}.$$

Therefore, $x$ is $(n+1)$-bi-obtainable.
\end{proof}

Let $p \in \N, p \geq 2$ and $A \subsetneq \Z(p)$ has more than $1$ element. In Proposition \ref{P1} (iv) it was pointed out that the set $A_p$ is a Cantor set. Every such a set will be called a $p$-Cantor set.
In the class of $p$-Cantor sets we will distinguish some special subclass. Let $l, r, p \in \N$, $p > 2$, $l + r < p$. We will consider Cantor sets of the form
$C(l, r, p) : = A(l, r, p)_p$
where
$$A(l, r, p): = \langle 0, l - 1\rangle \cup \langle p-r, p - 1\rangle .$$
We call such a set a special $p$-Cantor set or, in short, an S-Cantor set. This will not lead to any misunderstandings because $p$ will be always established.
A set $C(l, r, p)$ has the following construction. In the first step, we divide $[0,1]$ into $p$ subintervals with equal lengths and we enumerate them (starting from $0$). Then we remove $p-l-r$ consecutive intervals starting from the one with number $l$. So, there remains $l$ intervals on the left and $r$ intervals on the right of the emergent gap. We continue this procedure for the remaining intervals. Note that, S-Cantor sets do not have to be symmetric but are always self-similar. However, we will also examine symmetric S-Cantor sets.

Let us introduce some more notation. If $A \subset \Z, A \neq \emptyset,$ then we write
$$diam(A) := \sup\{|a - b|: a, b \in A\},$$
$$\Delta(A) := \sup\{b-a: a, b \in A, a<b, (a, b) \cap A = \emptyset\},$$
$$I(A):= \frac{\Delta(A)}{\Delta(A)+ diam(A)}.$$

\begin{theorem} \label{BBFS} \cite{BBFS}
Let $A \subset \Z$ be a nonempty finite set, $p \in \N$, $p \geq 2$. Then $A_p$ is an interval if and only if $\frac{1}{p} \geq I(A)$.
\end{theorem}
\begin{remark}
The assertion of the above theorem is equivalent to: $A_p$ is an interval if and only if $p \leq 1 + \frac{diam(A)}{\Delta(A)}$.
\end{remark}
\begin{corollary} \label{wBBFS}
Let $p \in \N,$ $p > 2,$ $A, B \subset \Z(p)$, $0, p - 1 \in A$ and $0, p - 1 \in B.$ Then
$A_p - B_p = [-1, 1]$ if and only if $\Delta(A - B) \leq 2.$
\end{corollary}
\begin{proof}
It is easily seen that $diam(A- B) = 2p - 2.$ Hence the inequality $p \leq 1+\frac{diam(A-B)}{\Delta(A-B)}$ is equivalent to $p-1 \leq \frac{2p-2}{\Delta(A-B)},$ and so to $\Delta(A - B) \leq 2.$ 
\end{proof}
Using the above corollary and Lemma \ref{Lpom}, we can prove the main theorem of this section, which gives a full characterization of the sets of the form $C(l_1,r_1, p) - C(l_2, r_2, p)$. Some inequalities with parameters $p,l_1,r_1,l_2,r_2$ will play the key role in this theorem. We will assume that these numbers are natural such that $p > 2$, $l_1+r_1 < p$, $l_2+r_2 < p$. Consider the following conditions:
\begin{equation*}
(S1)\,\,\, l_1+l_2+r_2 \geq p \;\text{ or }\; l_1+r_1+r_2 \geq p;
\end{equation*}%
\begin{equation*}
(S2)\,\,\, l_1+r_1+l_2 \geq p \;\text{ or }\; r_1+l_2+r_2 \geq p;
\end{equation*}
\begin{equation*}
(S3)\,\,\, l_1+r_1+l_2+r_2 \leq p;
\end{equation*}%
\begin{equation*}
(S1^* )\,\,\, l_1+l_2+r_2 > p \;\text{ or }\; l_1+r_1+r_2 > p;
\end{equation*}%
\begin{equation*}
(S2^*)\,\,\, l_1+r_1+l_2 > p \;\text{ or }\; r_1+l_2+r_2 > p.
\end{equation*}
Obviously $(S1)$ follows from $(S1^*)$ and $(S2)$ from $(S2^*)$. Moreover, if $(S3)$ holds, then $(S1)$ and $(S2)$ do not hold (so $(S1^*)$ and $(S2^*)$ do not hold as well). 

In the theorem we will consider the following combinations of conditions:
\begin{itemize}
\item[(1)] $(S1) \wedge (S2)$

\item[(2)] $(S3)$

\item[(3)] $(S1^*) \wedge \neg(S2)$

\item[(4)] $(S2^*) \wedge \neg(S1)$

\item[(5)] $\neg(S1^*) \wedge \neg(S2^*) \wedge \neg(S3) \wedge \neg((S1)\wedge (S2)).$
\end{itemize}
Denoting by $A_i$, $i = 1, \dots,5,$ the sets of admissible parameters determined
by the conditions (1), (2), $\dots$, (5), respectively, it is easy to see that the sets $A_i$,
$i = 1, \dots , 4$, are pairwise disjoint and that $A_5$ is the complement of their union. So, for any parameters exactly one of the conditions (1)--(5) holds.

\begin{remark} \label{min}
Observe that, when we consider the conditions $(S1),(S1^*),(S2),(S2^*),(S3)$ for the set
$$C(l_2,r_2,p)-C(l_1,r_1,p) = -( C(l_1,r_1,p)-C(l_2,r_2,p)),$$
we need to swap the indices "$1$" and "$2$" in the original inequalities. This means that conditions $(S1),(S1^*),(S2),(S2^*)$ for the set $C(l_2,r_2,p)-C(l_1,r_1,p)$ are the same as conditions $(S2),(S2^*),(S1),$ and $(S1^*)$, respectively, for the set $C(l_1,r_1,p)-C(l_2,r_2,p)$, while condition $(S3)$ does not change. In particular, $(4)$ holds for $C(l_1,r_1,p)-C(l_2,r_2,p)$ if and only if $(3)$ holds for $C(l_2,r_2,p)-C(l_1,r_1,p)$.
\end{remark}
\begin{remark} \label{idea}
The proof of the next theorem is quite technical. That is why we would like to present first a geometric idea of the proof. 
As will be shown in the proof, the set $C(l_1,r_1,p)-C(l_2,r_2,p)$ can be presented as $D_p$ for some set $D = \langle -p+1, p-1 \rangle \setminus (L \cup R)$ where $L = \langle a,b \rangle,$ $R = \langle c, d\rangle,$ $a,b \in \{-p+2, \dots, -1\},$ $c,d \in \{ 1, \dots, p-2\}$. We will briefly describe the construction of the set $D_p$ without going into details. In the first step of the construction, we divide the interval $[-1,1]$ into $2p$ intervals with the same length (equal to $\frac{1}{p}$). Endpoints of these intervals are of the form $\frac{i}{p}$ for $i \in \langle -p ,p \rangle$. We delete the intervals $(\frac{i-1}{p},\frac{i}{p}]$ for $i\in L$ and the intervals $[\frac{i}{p},\frac{i+1}{p})$ for $i\in R$. Denote the remaining set by $F_1$, that is, $$F_1 = [-1,1] \setminus \left(\bigcup_{i\in L} \left(\frac{i-1}{p},\frac{i}{p}\right] \cup \bigcup_{i \in R} \left[\frac{i}{p},\frac{i+1}{p}\right)\right).$$ In the second step of the construction, we repeat this procedure for every closed interval of the form $[\frac{i-1}{p},\frac{i+1}{p}]$ where $i \in D$. We denote the obtained set (which is a copy of $F_1$ with center at $i$ and the diameter equal to $\frac{2}{p}$) by $F^i_2$. Then we put $F_2 := \bigcup_{i\in D} F^i_2$. Note that some gaps that appeared in the first step of the construction could be covered in the next steps. Generally, for any $n \in \N$ we divide all of the intervals of the form $[x-\frac{1}{p^n},x+\frac{1}{p^n}]$, where $x=\sum_{i=1}^n \frac{x_i}{p^i}$, $x_i \in D$ for $i \in \{1,\dots,n\}$, into $2p$ intervals with the same length (equal to $\frac{1}{p^{n+1}}$). Then we delete the intervals $(x +\frac{i-1}{p^{n+1}},x+\frac{i}{p^{n+1}}]$ for $i\in L$ and the intervals $[x+\frac{i}{p^{n+1}},x+\frac{i+1}{p^{n+1}})$ for $i\in R$. We denote the obtained set by $F^x_{n+1}$, so
$$F_{n+1}^x = \left[x-\frac{1}{p^n},x+\frac{1}{p^n}\right] \setminus \left(\bigcup_{i\in L} \left(x +\frac{i-1}{p^{n+1}},x+\frac{i}{p^{n+1}}\right] \cup \bigcup_{i \in R} \left[x+\frac{i}{p^{n+1}},x+\frac{i+1}{p^{n+1}}\right)\right).$$ The union of all received sets in that way is denoted by $F_{n+1}$. At the end, we obtain $D_p = \bigcap_{n\in\N} \bigcup_{j\geq n} F_j$. 

The picture below shows two first steps of the construction of the set $\{-4,0,2,3,4\}_5$. The first set is $F_1$. In the next three rows we present all copies of $F_1$ obtained in the second step of the construction. The final picture is the set $F_2$. A line at the end of an interval means that the endpoint of this interval belongs to the set.
\begin{figure}[h]
\caption{Two first steps of the construction of the set $\{-4,0,2,3,4\}_5$}
\includegraphics[width=1\textwidth]{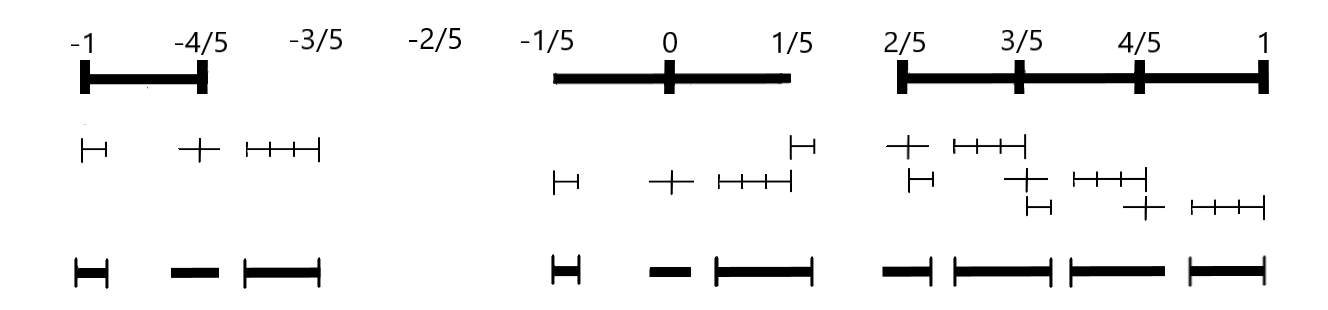}
\end{figure}

Observe that if $|L|,|R| \leq 1$ (which holds if $(S1)$ and $(S2)$ are satisfied), then the gaps that appear in one step are covered by the intervals obtained in the next steps. That is why $D_p = [-1,1]$. In the proof, we will also show that if $R\cap (p+L) \neq \emptyset$ (which is implied by $(S3)$), then in every interval from the construction, we will find a gap, which implies that $D_p$ is a Cantor set. If $|L| = 0$ (which holds if $(S1^*)$ is satisfied) and $|R| > 1$, then gaps will only appear near $1$ and near the left endpoints of other gaps. So, then we obtain an L-Cantorval. Similarly, we receive an R-Cantorval. Finally, if $|L| > 0$, $|R| > 0$, $R\cap (p+L) = \emptyset$ and  $|L| > 1$ or $|R| >1$, then gaps will only appear near $0,1$ and near both endpoints of other gaps. Hence we obtain an M-Cantorval. All details are included in the proof of the next theorem.
\end{remark}
\begin{theorem} \label{twR}
Let $l_1,r_1,l_2,r_2,p \in \N$, $p > 2$, $l_1+r_1 < p$, $l_2+r_2 < p$.

\begin{itemize}
\item[(1)] $C(l_1,r_1, p) - C(l_2, r_2, p) = [-1, 1]$
if and only if $(S1)$ and $(S2)$ hold.

\item[(2)] $C(l_1,r_1, p) - C(l_2, r_2, p)$ is a Cantor set if and only if
$(S3)$ holds.

\item[(3)] $C(l_1,r_1, p) - C(l_2, r_2, p)$ is an L-Cantorval if and only if $(S1^* )$ holds, but $(S2)$ does not hold.

\item[(4)] $C(l_1,r_1, p) - C(l_2, r_2, p)$ is an R-Cantorval if and only if $(S2^* )$ holds, but $(S1)$ does not hold.

\item[(5)] $C(l_1,r_1, p) - C(l_2, r_2, p)$ is an M-Cantorval if and only if $(S1^*)$, $(S2^*)$, $(S3)$ do not hold and at least one of $(S1)$ or $(S2)$ also does not hold.
\end{itemize}
\end{theorem}
\begin{proof}
Since for any parameters exactly one of the conditions given in (1)--(5) holds and a set cannot have at the same time two of the following forms: the interval $[-1, 1]$, a Cantor set, an L-Cantorval, an R-Cantorval or an M-Cantorval, we only need to prove the implications "$\Leftarrow$". 

Let us put
$$A := A(l_1, r_1, p) = \langle 0, l_1 - 1\rangle \cup \langle p-r_1, p - 1\rangle,$$
$$B := A(l_2, r_2, p) = \langle 0, l_2 - 1\rangle \cup \langle p-r_2, p - 1\rangle .$$
Then $A - B$ is equal to
$$\langle -p + 1,\; l_1   +r_2- p-1\rangle\;  \cup \;\langle -l_2 + 1,\; l_1 - 1\rangle\; \cup\;\langle -r_1+ 1,\; r_2-1\rangle\; 
\cup\;\langle p-r_1-l_2+1,\; p - 1\rangle .$$
Observe that the second and the third component contain $0$. Hence $A -B$ is equal to
$$\langle -p + 1,\; l_1  +r_2- p-1\rangle\; \cup\; \langle \min\{-l_2 , -r_1\}+ 1,\; \max\{l_1, r_2\}-1\rangle \;\cup\; \langle p-r_1-l_2+1,\; p - 1\rangle.$$

Let $$L = \langle l_1+r_2-p, \min\{-l_2, -r_1\}\rangle,$$ $$R= \langle \max\{l_1, r_2\}, p-r_1-l_2 \rangle.$$ Then $A-B = \langle -p+1, p-1 \rangle \setminus (L \cup R)$. 

Observe that $L$ has at most one element if and only if $l_1+r_2-p \geq -l_2 \;\text{ or }\; l_1+r_2-p\geq -r_1$ and $R$ has at most one element if and only if $p-r_1-l_2\leq l_1  \;\text{ or }\;  p-r_1-l_2\leq r_2$. So, we have

\begin{equation} \label{o1}
 \left\{ \begin{array}{ccc}
(S1) \Leftrightarrow |L| \leq 1 \Leftrightarrow l_1+r_2-p+1 \notin L \\ 
(S1^*) \Leftrightarrow L = \emptyset \Leftrightarrow l_1+r_2-p \notin L\\
(S2) \Leftrightarrow |R| \leq 1 \Leftrightarrow p-r_1-l_2-1 \notin R \\
(S2^*) \Leftrightarrow R = \emptyset \Leftrightarrow p-r_1-l_2 \notin R
\end{array}%
\right. 
\end{equation}
Moreover, since $p+\min\{-l_2, -r_1\} > p-r_1-l_2$ and
$$R \cap (p+L) = \langle \max\{l_1, r_2\}, p-r_1-l_2 \rangle \cap \langle l_1+r_2, p+ \min\{-l_2, -r_1\}\rangle,$$
we have 
\begin{equation}\label{o2}
(S3)\Leftrightarrow R \cap (p+L) \neq \emptyset\Leftrightarrow p-r_1-l_2, l_1+r_2 \in R \cap (p+L) .
\end{equation}

Note that $-p+1,0,p-1 \in A-B$, so the assumptions of Lemma \ref{Lpom} are satisfied. Observe also that if $R \cap (p+L) = \emptyset$, then $k \in A-B$ or $k-p \in A-B$ for any $k \in \langle 0, p-1 \rangle$. Indeed, if $k \notin A-B$, then $k \in R$, and thus $0>k-p \notin L$, so $k-p \in A-B$. Hence if $(S3)$ does not hold, then the assumptions of Lemma \ref{Lpom} (4) are satisfied.

By Proposition \ref{dod} we have $$C(l_1,r_1, p) - C(l_2, r_2, p) = A_p - B_p = (A-B)_p.$$ Since $A-B$ has more than one element, $(A-B)_p$ is perfect, by Proposition \ref{P1}.

Ad (1) From Corollary \ref{wBBFS} it follows that the equality $C(l_1,r_1, p) - C(l_2, r_2, p) = [-1, 1]$ is equivalent to the inequality $\Delta(A - B) \leq 2$, which holds if and only if $L$ and $R$ have at most one element. And this is satisfied if and only if $(S1)$ and $(S2)$ hold.

Ad (2) 
We have to show that $(A-B)_p$ is nowhere dense. Let $x\in (A-B)_p$, $(x_i)\in (A-B)^{\N}$ be such that $x= \sum_{i=1}^{\infty} \frac{x_i}{p^i}$ and $\ve > 0$. We will find an interval $(y,z) \subset (x- \ve,x+\ve)$, which is disjoint with $(A-B)_p$. Let $n \in \N$ be such that $\frac{2}{p^n} < \ve$. Define $y,z \in [0,1]$ in the following way: 
$$y_i = x_i = z_i\mbox{ for }i \leq n,$$ 
$$y_{n+1} = z_{n+1} = p-r_1-l_2-1, $$ 
$$y_{n+2}=p-r_1-l_2-2, z_{n+2}=y_{n+2}+3=p-r_1-l_2+1,$$ 
$$y_i = p-1, z_i = -p+1\mbox{ for }i > n+2$$
and let $y = \sum_{i=1}^{\infty} \frac{y_i}{p^i}$, and $z = \sum_{i=1}^{\infty} \frac{z_i}{p^i}$. Then $y_i,z_i \in \langle -p+1,p-1 \rangle$ for $i \in \N$. Hence $$z-y = \frac{3}{p^{n+2}}-\sum_{i=n+3}^{\infty}\frac{2p-2}{p^i}=\frac{1}{p^{n+2}}>0.$$

We have $|x-y| \leq \sum_{i=n+1}^{\infty} \frac{2p-2}{p^i} = \frac{2}{p^n} < \ve.$ Similarly, $|z-x| < \ve$, and thus $(y,z) \subset (x-\ve, x+\ve)$. Let $w \in (y,z)$. We will show that $w \notin (A-B)_p$. On the contary, assume that $w \in (A-B)_p$. So, there is a sequence $(w_i) \in (A-B) ^{\N}$ such that $w = \sum_{i=1}^{\infty} \frac{w_i}{p^i}$. 
Since 
$$y_{n+2} =p-r_1-l_2-2\geq p-r_1-l_2-l_1-r_2\geq 0,$$ 
we have 
$$y = \overline{y}_{n+1} + \frac{y_{n+2}}{p^{n+2}} + \sum_{i=n+3}^{\infty}\frac{p-1}{p^i} > \overline{y}_{n+1}$$ and 
$$z = \overline{y}_{n+1} + \frac{p-r_1-l_2+1}{p^{n+2}} - \sum_{i=n+3}^{\infty}\frac{p-1}{p^i}< \overline{y}_{n+1} + \frac{p}{p^{n+2}} < \overline{y}_{n+1} + \frac{1}{p^{n+1}},$$ 
and so $w \in (\overline{y}_{n+1}, \overline{y}_{n+1}+\frac{1}{p^{n+1}}).$ By Lemma \ref{Lpom} (1), $\overline{w}_{n+1} = \overline{y}_{n+1}$ or $\overline{w}_{n+1} = \overline{y}_{n+1} + \frac{1}{p^{n+1}}$. 
Consider the cases. 

$1^{\mbox{o}}$  $\overline{w}_{n+1} = \overline{y}_{n+1}$.
We have 
$$w \in (y,z)=\left(\overline{y}_{n+2}+\sum_{i=n+3}^{\infty} \frac{p-1}{p^i},\overline{z}_{n+2}-\sum_{i=n+3}^{\infty} \frac{p-1}{p^i}\right)= \left(\overline{y}_{n+2}+\frac{1}{p^{n+2}},\overline{z}_{n+2}-\frac{1}{p^{n+2}}\right)$$$$ = \left(\overline{y}_{n+2}+\frac{1}{p^{n+2}},\overline{y}_{n+2}+\frac{3}{p^{n+2}}-\frac{1}{p^{n+2}}\right)=\left(\overline{y}_{n+2}+\frac{1}{p^{n+2}},\overline{y}_{n+2}+\frac{2}{p^{n+2}}\right).$$ 
Using again Lemma \ref{Lpom} (1), we have $$w_{n+2} = y_{n+2}+1 = p-r_1-l_2-1$$ or $$w_{n+2}= y_{n+2}+2 = p-r_1-l_2.$$ Since $(S2)$ does not hold, by (\ref{o1}), we obtain $|R|>1$, and consequently $\{p-r_1-l_2-1, p-r_1-l_2\} \subset R$. Thus, $w_{n+2} \in R$, a contradiction. 

$2^{\mbox{o}}$  $\overline{w}_{n+1} = \overline{y}_{n+1}+ \frac{1}{p^{n+1}}$.
Since $y_{n+1} + 1> 0$, by Lemma \ref{Lpom} (1.1), we have $$w_{n+1} = y_{n+1}+1 = p-r_1-l_2 \in R$$ or $$w_{n+1}= -p+y_{n+1}+1= -p+p-r_1-l_2.$$ Since $p-r_1-l_2 \in R$ and $w_{n+1} \in A-B$, we must have $w_{n+1}=-r_1-l_2$, but, because $(S3)$ holds, by (\ref{o2}), $p-r_1-l_2 \in (L + p)$, and hence $-r_1-l_2\in L$. So, $w_{n+1} \notin (A-B)$, a contradiction.
Thus, $(y,z) \in (x- \ve, x+\ve) \setminus (A-B)_p$, so $(A-B)_p$ is nowhere dense. It is also perfect, so it is a Cantor set.

Ad (3)--(5)
First, we will show some additional conditions $(C1),(C2)$ and $(C3)$.

$(C1)$  \,\,\,    If conditions $(S2)$ and $(S3)$ do not hold, then every gap in $(A-B)_p$ is accumulated on the left by infinitely many gaps and proper components of $(A-B)_p$. 

First, observe that $1 \in A-B$. Indeed, if $1 \notin A-B$, then $1 \in R$, and so $l_1=r_2=1$. Since $(S2)$ does not hold, we have $r_1+l_2 < p-1$. Hence $r_1+l_2+l_1+r_2 < p+1,$ i.e. $(S3)$ holds, a contradiction. Thus, $1 \in A-B$.

Now, let $x$ be the left endpoint of a gap in $(A-B)_p$ and $(x_i) \in (A-B)^{\N}$ be such that $x = \sum_{i=1}^{\infty} \frac{x_i}{p_i}$. Then, by Lemma \ref{Lpom} (2), there is $n \in \N$ such that 
$$x_i = p-1 \mbox{ for } i>n.$$ 
Let $\ve > 0$ and $m>n$ be such that $\frac{1}{p^m} < \ve$. We will find intervals $[y,z] \subset (x- \ve,x) \cap (A-B)_p$ and $(s,t) \subset (x-\ve,x) \setminus (A-B)_p$. Define $y,z,s,t$ in the following way:
$$y_i=z_i=s_i=t_i = x_i\mbox{ for }i \leq m,$$ 
$$y_{m+1} = 0, z_{m+1}=1, s_{m+1} = p-r_1-l_2-2, t_{m+1} =s_{m+1}+3= p-r_1-l_2+1,$$ 
$$y_i = z_i =0, s_i =x_i= p-1, t_i =-p+1\mbox{ for }i > m+1$$
and put $y = \sum_{i=1}^{\infty} \frac{y_i}{p_i},z = \sum_{i=1}^{\infty} \frac{z_i}{p_i},t = \sum_{i=1}^{\infty} \frac{t_i}{p_i},s = \sum_{i=1}^{\infty} \frac{s_i}{p_i}$. Observe
that $z>y$, $x>t$ and 
\begin{equation*}
t-s=\frac{3}{p^{m+1}}-\sum_{i=m+2}^{\infty }\frac{2p-2}{p^{i}}=\frac{1}{%
p^{m+1}}>0.
\end{equation*}%
Since $1\in A-B$, we have $1\notin R$, so every element of $R$ is greater than $1$. Because $(S2)$ does not hold, from (\ref{o1}) we obtain $%
p-r_{1}-l_{2}-1\in R$. Hence $p-r_{1}-l_{2}-1 > 1$, and therefore
\begin{equation*}
s_{m+1} = p-r_{1}-l_{2}-2 \geq 1=z_{m+1}.
\end{equation*}%
Consequently, $s>z$. Thus,
\begin{equation*}
0<x-t<x-s<x-z<x-y=\sum_{i=m+1}^{\infty }\frac{p-1}{p^{i}}=\frac{1}{p^{m}}%
<\varepsilon ,
\end{equation*}%
which implies $\left[ y,z\right] \subset \left( x-\varepsilon ,x\right) $
and $\left( s,t\right) \subset \left( x-\varepsilon ,x\right) $.

Now, we will show that $\left[ y,z\right] \subset \left( A-B\right) _{p}$.
Let $w\in \left[ y,z\right] $. Then%
\begin{equation*}
w\in \left[ y,z\right] =\left[ \overline{y}_{m+1},\overline{z}_{m+1}\right] =%
\left[ \overline{y}_{m+1},\overline{y}_{m+1}+\frac{1}{p^{m+1}}\right] .
\end{equation*}
Since $y_{m+1} = 0 \in (A-B), z_{m+1}=1 \in (A-B)$ and $y_i=z_i =x_i \in (A-B)$ for $i \leq m$, the point $w$ is $(m+1)$-bi-obtainable. Since $(S3)$ does not hold, by (\ref{o2}), we have $R \cap (p+L) = \emptyset.$ By Lemma \ref{Lpom} (4), $w$ is $(i)$-bi-obtainable for $i > m+1$, and by Lemma \ref{Lpom} (3), $w \in (A-B)_p,$ and thus $[y,z] \subset (A-B)_p.$ 

Now, we will show that $(s,t) \cap (A-B)_p = \emptyset$. Let $v \in (s,t)$. Assume on the contrary that there exists a sequence $(v_i) \in (A-B)^{\N}$ such that $v =\sum_{i=1}^{\infty} \frac{v_i}{p^i}$. Since $s_{m+1}\geq 1>0$ and $s_{i}=p-1>0$ for $i>m+1$, we have $\overline{s}_{m}<s$. Thus  $\overline{s}_{m}<s<t<x$, and so 
\begin{equation*}
v\in \left( s,t\right) \subset \left(\overline{s}_m, \overline{s}_m + \frac{1}{p^m}\right)=\left(\overline{s}_m, \overline{s}_m + \sum_{i=m+1}^{\infty}\frac{p-1}{p^i}\right)=\left(\overline{s}_m, \overline{x}_m + \sum_{i=m+1}^{\infty}\frac{x_i}{p^i}\right)= \left( \overline{s}_{m},x\right) .
\end{equation*}

Using Lemma \ref{Lpom} (1), we have $ \overline{v}_m = \overline{s}_m $ or $\overline{v}_m  = x$.
If $\overline{v}_m = x$,  
by Lemma \ref{Lpom} (2'), $x$ is not an endpoint of a gap, a contradiction.
So, $\overline{v}_m = \overline{s}_m $. 
We have 
$$v \in (s,t) =\left(\overline{s}_{m+1} + \sum_{i=m+2}^{\infty} \frac{p-1}{p^i}, \overline{t}_{m+1} -\sum_{i=m+2}^{\infty} \frac{p-1}{p^i}\right)= \left(\overline{s}_{m+1}  + \frac{1}{p^{m+1}}, \overline{t}_{m+1}  - \frac{1}{p^{m+1}}\right)$$$$=\left(\overline{s}_{m+1}  + \frac{1}{p^{m+1}}, \overline{s}_{m+1}+\frac{3}{p^{m+1}}  - \frac{1}{p^{m+1}}\right)=\left(\overline{s}_{m+1}  + \frac{1}{p^{m+1}}, \overline{s}_{m+1}+\frac{2}{p^{m+1}}\right).$$ 
Using again Lemma \ref{Lpom} (1), we have $$v_{m+1}= s_{m+1}+1 = p-r_1-l_2-1$$ or $$v_{m+1} =s_{m+1}+2= p-r_1-l_2.$$ Since $(S2)$ does not hold, by (\ref{o1}), we obtain $|R| > 1,$ and consequently $\{ p-r_1-l_2-1, p-r_1-l_2 \} \subset R$. Hence $v_{m+1} \in R$, a contradiction. 
Thus, $(s,t) \subset (x-\ve ,x) \setminus (A-B)_p$, which finishes the proof of $(C1)$. 

The following (symmetric to $(C1)$) condition follows directly from $(C1)$ and Remark \ref{min}. 

$(C2)$  \,\,\,    If conditions $(S1)$ and $(S3)$ do not hold, then every gap in $(A-B)_p$ is accumulated on the right by infinitely many gaps and proper components of $(A-B)_p$. 

Now, we will prove the last condition $(C3)$.

$(C3)$  \,\,\,    If $(S1^*)$ holds, but condition $(S2)$ does not hold, then every gap in $(A-B)_p$ has an adjacent interval (contained in $(A-B)_p$) on the right. 

Let $x$ be the right endpoint of a gap in $(A-B)_p$ and $(x_i) \in (A-B)^{\N}$ be such that $x = \sum_{i=1}^{\infty} \frac{x_i}{p_i}$. By Lemma \ref{Lpom} (2), there is $n \in \N$ such that $x_i = -p+1$ for $i>n$. Let $$y_i = x_i\mbox{ for }i \leq n,$$ $$y_i = 0\mbox{ for }i > n.$$ Put $y = \sum_{i=1}^{\infty} \frac{y_i}{p^i}$. We will show that $[x,y] \subset (A-B)_p$. Let $w \in [x,y]$. 
Since $$w \in [x,y] = \left[y-\sum_{i=n+1}^{\infty} \frac{p-1}{p^i},y\right]= \left[y-\frac{1}{p^{n}},y\right],$$ there is $q \in [-\frac{1}{p^n},0]$ such that $w = y+q$. By Observation \ref{wiel} (v), there is a sequence $(w_i)_{i=n+1}^{\infty} \in \langle -p+1, 0 \rangle ^{\N}$ such that $\sum_{i=n+1}^{\infty} \frac{w_i}{p^i} = q$, and so $y+ \sum_{i=n+1}^{\infty} \frac{w_i}{p^i} = w$. Since $(S1^*)$ holds, by (\ref{o1}), we have $L = \emptyset$. Therefore, $\langle -p+1, 0 \rangle \subset A-B$, and so $w_i \in A-B$ for $ i > n$. For $i \leq n$ put $w_i = y_i \in A-B$. We found a sequence $(w_i) \in (A-B)^{\N}$ such that $w= \sum_{i=1}^{\infty} \frac{w_i}{p^i}.$ Thus, $w \in (A-B)_p$, which finishes the proof of $(C3)$.

Now, we can prove (3). Assume that $(S1^* )$ holds, but $(S2)$ does not hold. Then $(S3)$ also does not hold and $(A-B)_p$ is not an interval, by (1), so it has a gap. Hence the assumptions of $(C1)$ and $(C3)$ are satisfied. Therefore, by $(C1)$, every gap in $(A-B)_p$ is accumulated on the left by infinitely many gaps and proper components of $(A-B)_p$, and by $(C3)$, every gap has an adjacent interval on the right. Thus, $(A-B)_p$ is an L-Cantorval, which finishes the proof of (3).

Since minus L-Cantorval is an R-Cantorval, by Remark \ref{min}, (4) follows from (3).

Now, we prove (5). Assume that $(S1^*)$, $(S2^*)$, $(S3)$ do not hold and at least one of $(S2)$ or $(S1)$ also does not hold.
If $(S2)$ does not hold, then, by $(C1)$, every gap in $(A-B)_p$ is accumulated on the left by infinitely many gaps and proper intervals contained in $(A-B)_p$. If $(S1)$ also does not hold, then, from $(C2)$ we infer that every gap in $(A-B)_p$ is accumulated on the right by infinitely many gaps and intervals, and thus we get the assertion.

Assume that $(S1)$ holds, but $(S2)$ does not hold. The other case again easily follows from Remark \ref{min} and the fact that minus M-Cantorval is an M-Cantorval.
First, we will show that $-1 \in A-B$ or $-p+2 \in A-B$. If $-1 \notin A-B$, then $-1 \in L$, and hence $r_1=l_2=1$. Since $(S2)$ does not hold, we have $l_1+r_1+l_2<p$, and thus $l_1 < p-2$. Because $l_1 \geq 1$, we have $p > 3$. From the fact that $(S1)$ holds, but $(S1^*)$ does not, it follows that $l_1+r_2=p-1$. We have $$l_1+r_2-p-1 = -2 \geq -p+2.$$ 
The fact that $(S1^*)$ does not hold and $-1 \in L$ together with $(\ref{o1})$ gives us that $L=\{-1\}.$ So, $-p+2 \in A-B$. Thus, we proved that $-1 \in A-B$ or $p-2 \in A-B$.

We will assume that $-1 \in A-B$. If $p-2 \in A-B$, then the proof is similar, with small difference, which we will point out.

Let $x$ be the right endpoint of a gap in $(A-B)_p$ and $(x_i) \in (A-B)^{\N}$ be such that $x = \sum_{i=1}^{\infty} \frac{x_i}{p_i}$. Then, by Lemma \ref{Lpom} (2), there is $n \in \N$ such that $x_i =-p+1$ for $i>n$. Let $\ve > 0$ and $m>n$ be such that $\frac{1}{p^m} < \ve$. We will find intervals $[y,z] \subset (x ,x+\ve) \cap (A-B)_p$ and $(s,t) \subset (x ,x+\ve) \setminus (A-B)_p$. Define $y,z,s,t$ in the following way:
$$y_i=z_i=s_i=t_i = x_i\mbox{ for }i \leq m,$$ 
$$y_{m+1} = -1, z_{m+1}=0, s_{m+1}=t_{m+1} = l_1+r_2-p-1,$$
$$y_{m+2} = z_{m+2}=0, s_{m+2} = p-r_1-l_2-2, t_{m+2} =s_{m+2}+3= p-r_1-l_2+1,$$ 
$$y_i = z_i =0, s_i = p-1, t_i =-p+1\mbox{ for }i > m+1$$ 
and put $y = \sum_{i=1}^{\infty} \frac{y_i}{p_i},z = \sum_{i=1}^{\infty} \frac{z_i}{p_i},t = \sum_{i=1}^{\infty} \frac{t_i}{p_i},s = \sum_{i=1}^{\infty} \frac{s_i}{p_i}$. 
We will show that $y,z,s,t \in (x,x+\ve)$. Of course, $y,z,s,t$ are greater than $x$. We have $$y <z = x + \sum_{i=m+1}^{\infty} \frac{p-1}{p^i} = x + \frac{1}{p^m} < x + \ve$$ and also
$$t-s = \frac{3}{p^{m+2}} - \sum_{i=m+3}^{\infty}\frac{2p-2}{p^i} = \frac{1}{p^{m+2}} > 0,$$
so $s < t$. 
Moreover, since $(S1^*)$ does not hold, by (\ref{o1}), $l_1+r_2-p \in L$, and in consequence $$l_1+r_2-p \leq 0,$$ and thus 
$$t= \overline{z}_m + \frac{l_1+r_2-p-1}{p^{m+1}} + \frac{p-r_1-l_2+1}{p^{m+2}} - \sum_{i=m+3}^{\infty} \frac{p-1}{p^i} \leq z + \frac{-1}{p^{m+1}} + \frac{p-r_1-l_2}{p^{m+2}} < z + \frac{-1}{p^{m+1}} + \frac{p}{p^{m+2}} = z.$$
Hence $s < t < z < x+\ve$, which ends the proof that $y,z,s,t \in (x,x+\ve)$.

We will show that $[y,z] \subset (A-B)_p$. Let $w \in [y,z]$. Then $$w \in[y,z]= \left[\overline{z}_m -\frac{1}{p^{m+1}},\overline{z}_m\right] = [\overline{y}_{m+1},\overline{z}_{m+1}].$$ Since $y_{m+1} = -1 \in A-B, z_{m+1}=0 \in A-B$ and $y_i=z_i=x_i \in A-B$ for $i \leq m$, the point $w$ is $(m+1)$-bi-obtainable (if $-1 \notin A-B$ it suffices to change $y_{m+1} = p-1, z_{m+1} = p-2$). Since $(S3)$ does not hold, by (\ref{o2}), we have $R \cap (p+L) = \emptyset.$ By Lemma \ref{Lpom} (4), $w$ is $(i)$-bi-obtainable for $i > m+1$, so, by Lemma \ref{Lpom} (3), $w \in (A-B)_p,$ and thus $[y,z] \subset (A-B)_p \cap (x, x+\ve).$ 

Now, we will show that $(s,t) \cap (A-B)_p = \emptyset$. Let $v \in (s,t)$. Assume on the contrary that there exists a sequence $(v_i) \in (A-B)^{\N}$ such that $v= \sum_{i=1}^{\infty} \frac{v_i}{p^i}$. Since $(S2)$ does not hold, by (\ref{o1}), we have $p-r_1-l_2-1 \in R$, and so $p-r_1-l_2-1 \geq 0$. Therefore, 
$$ s = \overline{s}_m + \frac{l_1+r_2-p-1}{p^{m+1}}+\frac{p-r_1-l_2-2}{p^{m+2}} + \sum_{i=m+3}^{\infty}\frac{p-1}{p^i}$$$$ \geq \overline{s}_m + \frac{l_1+r_2-p-1}{p^{m+1}} + \frac{-1}{p^{m+2}} + \frac{1}{p^{m+2}} > \overline{s}_m - \frac{p}{p^{m+1}} = \overline{s}_m - \frac{1}{p^{m}}$$
and
$$t = \overline{s}_m + \frac{l_1+r_2-p-1}{p^{m+1}}+\frac{p-r_1-l_2+1}{p^{m+2}} - \sum_{i=m+3}^{\infty}\frac{p-1}{p^i} $$$$< \overline{s}_m + \frac{l_1+r_2-p-1}{p^{m+1}}+\frac{p}{p^{m+2}} =\overline{s}_m + \frac{l_1+r_2-p}{p^{m+1}} \leq \overline{s}_m,$$
because $l_1+r_2-p \leq 0$.
Hence $$v \in \left(\overline{s}_m- \frac{1}{p^m}, \overline{s}_m\right) = \left(\overline{x}_m- \sum_{i=m+1}^{\infty}\frac{p-1}{p^i}, \overline{s}_m\right)= \left(\overline{x}_m+ \sum_{i=m+1}^{\infty}\frac{x_i}{p^i}, \overline{s}_m\right)=(x, \overline{s}_m).$$ 
Thus, using Lemma \ref{Lpom} (1), we obtain $ \overline{v}_m = \overline{s}_m$ or $\overline{v}_m = x$, but the second case is impossible (by Lemma \ref{Lpom} (2')). Hence $ \overline{v}_m = \overline{s}_m$. We also have 
$$ s = \overline{s}_{m+1} +\frac{p-r_1-l_2-2}{p^{m+2}} + \sum_{i=m+3}^{\infty}\frac{p-1}{p^i} \geq \overline{s}_{m+1} +\frac{-1}{p^{m+2}} + \frac{1}{p^{m+2}} = \overline{s}_{m+1}$$
and
$$t = \overline{s}_{m+1} +\frac{p-r_1-l_2+1}{p^{m+2}} - \sum_{i=m+3}^{\infty}\frac{p-1}{p^i} < \overline{s}_{m+1} +\frac{p}{p^{m+2}} =  \overline{s}_{m+1} +\frac{1}{p^{m+1}},$$
and therefore
 $$v \in \left(\overline{s}_{m+1}, \overline{s}_{m+1} + \frac{1}{p^{m+1}}\right).$$ Thus, by Lemma \ref{Lpom} (1), $\overline{v}_{m+1} = \overline{s}_{m+1}$ or $\overline{v}_{m+1} = \overline{s}_{m+1}+1$ and since $\overline{v}_m = \overline{s}_m$, we have ${v}_{m+1} = s_{m+1}$ or ${v}_{m+1} = s_{m+1}+1$. However, $$s_{m+1}+1 =l_1+r_2-p \notin A-B,$$ so $v_{m+1} = s_{m+1}$. Since
$$v \in (s,t) =\left(\overline{s}_{m+2} + \sum_{i=m+3}^{\infty} \frac{p-1}{p^i}, \overline{s}_{m+1} + \frac{s_{m+2}+3}{p^{m+2}} -\sum_{i=m+3}^{\infty} \frac{p-1}{p^i}\right)$$$$= \left(\overline{s}_{m+2} + \frac{1}{p^{m+2}}, \overline{s}_{m+2} + \frac{2}{p^{m+2}}\right),$$ 
 we have, by Lemma \ref{Lpom} (1), $$v_{m+2} = s_{m+2} + 1= p-r_1-l_2-1$$ or $$v_{m+2} = s_{m+2} + 2 = p-r_1-l_2,$$ but since $(S2)$ does not hold, we have $\{ p-r_1-l_2-1, p-r_1-l_2 \} \cap (A-B) = \emptyset$, a contradiction. Thus, $(s,t) \subset (x,x +\ve) \setminus (A-B)_p$. 
Finally, we obtain (5). 
%
%
\end{proof}

\begin{corollary} \label{wniosek1}
Let $l,r,p\in \N$, $p > 2$, $l+r < p$. Then 
\begin{itemize}
\item[(1)] $C(l, r, p) - C(l, r, p) = [-1, 1]$
if and only if 
\begin{equation*}
2l+r\geq p \;\text{ or }\; l+2r \geq p;
\end{equation*}%
\item[(2)] $C(l, r, p) - C(l, r, p)$ is a Cantor set if and only if 
\begin{equation*}
2l+2r \leq p;
\end{equation*}%
\item[(3)] $C(l, r, p) - C(l, r, p)$ is an M-Cantorval if and only if
\begin{equation*}
2l+r<p\;\text{ and }\; l+2r< p \text{ and }\; 2l+2r > p.
\end{equation*}%
\end{itemize}
\end{corollary}
\begin{proof}
Putting $l$ in the place of $l_1$ and $l_2$, and $r$ in the place of $r_1$ and $r_2$, in the conditions from Theorem \ref{twR}, we receive the assertion.
\end{proof}
\begin{example}
Let $r_1=1$, $l_1=1$, $r_2=1$, $l_2=2$, $p=4$. We will show that $C(l_1,r_1,p)-C(l_2,r_2,p) = [-1,1]$. Of course, $l_1+r_1<p, l_2+r_2<p$. Moreover, $l_1+l_2+r_2 \geq p$ and $l_1+l_2+r_1 \geq p$, and consequently conditions $(S1)$ and $(S2)$ are satisfied. So, $C(l_1,r_1,p)-C(l_2,r_2,p) = [-1,1]$. At the same time, $2l_1+2r_1 \leq p $ and $2l_2+r_2 \geq p$, thus, by Corollary \ref{wniosek1}, $C(l_2,r_2,p)-C(l_2,r_2,p) = [-1,1]$, but $C(l_1,r_1,p)-C(l_1,r_1,p)$ is a Cantor set.
\end{example}
\begin{example}
Let $r_1=2$, $l_1=3$, $r_2=3$, $l_2=1$, $p=7$. We will show that $C(l_1,r_1,p)-C(l_2,r_2,p)$ is an L-Cantorval and, by symmetry, $C(l_2,r_2,p)-C(l_1,r_1,p)$ is an R-Cantorval. We have $l_1+r_1<p, l_2+r_2<p.$ Moreover, $l_1+r_1+r_2 > p$, so $(S1^*)$ holds and $l_1+r_1+l_2 < p$, $r_1+l_2+r_2 < p$, therefore $(S2)$ does not hold. In consequence, $$C(l_1,r_1,p)-C(l_2,r_2,p)=\{-6,-5,-4,-3,-2,-1,0,1,2,5,6\}_7$$ is an L-Cantorval, and $$C(l_2,r_2,p)-C(l_1,r_1,p)=\{-6,-5,-2,-1,0,1,2,3,4,5,6\}_7$$ is an R-Cantorval. At the same time, $2l_1+r_1 \geq p$ and $2r_2+l_2 \geq p$, thus, by Corollary \ref{wniosek1}, $$C(l_1,r_1,p)-C(l_1,r_1,p)= C(l_2,r_2,p)-C(l_2,r_2,p)= [-1,1].$$
\end{example}
Let us introduce the notion of a symmetric S-Cantor set. Observe that $C(l, r, p)$ is symmetric with respect to $\frac{1}{2}$ if and only if $l=r$. Then the set $C(l, p) := C(l, l, p)$ is called a symmetric S-Cantor set. 
\begin{corollary}
Let $l_1,l_2,p \in \N,$ $p > 2,$ $2l_1 < p,$ $2l_2 < p$. Then 
\begin{itemize}
\item[(1)] $C(l_1, p) - C(l_2, p) = [-1, 1]$
if and only if 
\begin{equation*}
2l_1+l_2\geq p \;\text{ or }\; l_1+2l_2 \geq p;
\end{equation*}%
\item[(2)] $C(l_1, p) - C(l_2, p)$ is a Cantor set if and only if
\begin{equation*}
2l_1+2l_2 \leq p;
\end{equation*}%
\item[(3)] $C(l_1, p) - C(l_2, p)$ is an M-Cantorval if and only if
\begin{equation*}
2l_1+l_2<p \;\text{ and }\; l_1+2l_2 < p \text{ and }\; 2l_1+2l_2 > p.
\end{equation*}%
\end{itemize}
\end{corollary}
\begin{proof}
Putting $l_1$ in the place of $r_1$, and $l_2$ in the place of $r_2$ in Theorem \ref{twR}, we obtain the assertion.
\end{proof}
Using the above corollary, we can receive a result similar to the theorem obtained by Kraft (although with additional possibility).
In the paper \cite{K}, the author considered central Cantor sets generated by a constant $\alpha$. The construction of such sets is similar as the construction of the classical Cantor set. The difference is that the constant ratio of the length of a removed interval to the length of the interval from which we remove is equal to $\alpha$ and not necessarily $\frac{1}{3}.$ Kraft proved that the difference set of such Cantor set generated by $\alpha$ is equal to $[-1,1]$ if $\alpha \leq \frac{1}{3}$, and is a Cantor set if $\alpha > \frac{1}{3}$.
\begin{corollary} \label{poj}
Let $l, p \in \N$,  $p > 2$, $2l < p$. Then 
\begin{itemize}
\item[(1)] $C(l, p) - C(l, p) = [-1, 1]$
if and only if 
\begin{equation*}
\frac{l}{p} \geq \frac{1}{3};
\end{equation*}%
\item[(2)] $C(l, p) - C(l, p)$ is a Cantor set if and only if 
\begin{equation*}
\frac{l}{p} \leq \frac{1}{4};
\end{equation*}%
\item[(3)] $C(l, p) - C(l, p)$ is an M-Cantorval if and only if
\begin{equation*}
\frac{l}{p} \in \left(\frac{1}{4}, \frac{1}{3}\right).
\end{equation*}%
\end{itemize}
\end{corollary}
\begin{example}
Let $l_1=2$, $l_2=1$, $p=5$. Then $2l_1+l_2 \geq p$, so $C(l_1,p) - C(l_2,p) = [-1,1]$. Moreover, $\frac{l_1}{p} \geq \frac{1}{3}$ and $\frac{l_2}{p} <\frac{1}{4}$, therefore $C(l_1,p)-C(l_1,p) = [-1,1]$, but $C(l_2,p)-C(l_2,p)$ is a Cantor set.
\end{example}
\begin{example}
Let $l=1,p=3$. Observe that the set $C = C(l,p)$ is the classical Cantor ternary set. By Corollary \ref{poj}, we get the known result $C-C = [-1,1]$.
\end{example}
\begin{example} \label{przyklad}
Let $l=2$, $p=7$. Then $\frac{l}{p} \in (\frac{1}{4}, \frac{1}{3})$, and therefore 
$$C(l,p)-C(l,p) = \{-6,-5,-4,-1,0,1,4,5,6\}_7$$
 is an M-Cantorval.
\end{example}
\begin{remark}
The classical Cantor set is often defined as $\{0,2\}_3$, that is, the set of all numbers in $[0,1]$ that can be written in the ternary system using only zeros and twos. Above examples show that Theorem \ref{twR} gives us new examples of Cantorvals and Cantor sets that can be described in the similar manner. For example, the set $\{-6,-5,-4,-1,0,1,4,5,6\}_7$ from Example \ref{przyklad}, which is an M-Cantorval, can be described as the set of all numbers in $[-1,1]$ which can be written in the signed digit representation with base $7$ and the digit set $\{-6,-5,\dots,5,6\}$ without using $-3,-2,2$ and $3$.  
\end{remark}
\section*{Acknowledgments} The author would like to thank Tomasz Filipczak for many fruitful conversations during the preparation of the paper.

The author was supported by the GA \v{C}R project 20-22230L.

\end{document}